\documentclass[a4paper,11pt]{amsart}

\usepackage{latexsym}
\usepackage{amssymb}
\usepackage[all]{xy}

\newtheorem{thm}{Theorem}
\newtheorem{dfn}{Definition}
\newtheorem{lemma}{Lemma}
\newtheorem{cor}{Corollary}
\newtheorem{question}{Question}

\newcommand{\bb}{\mathsf{b}} 
 
\newcommand{\di}{\mathsf{d}} 
\newcommand{\M}{\mathcal{M}} 
\newcommand{\modu}{\mathrm{mod}}

\begin{document}

\title{On the Hochschild (Co)homology of\\ 
Quantum Homogeneous Spaces}

\author{Ulrich Kr\"ahmer} 

\address{University of Glasgow, Department of
Mathematics, University Gardens, 
Glasgow G12 8QW, UK}

\email{ukraehmer@maths.gla.ac.uk}

\begin{abstract}
The recent result of Brown and Zhang 
establishing Poincar\'e duality in the 
Hochschild (co)ho\-mo\-lo\-gy
of a large class of Hopf algebras
is extended to right coideal subalgebras  
over which the Hopf algebra is faithfully
flat, and applied to the standard Podle\'s 
quantum 2-sphere.
\end{abstract}
\maketitle

\section{Introduction} 
\subsection{Theory} 
As work in particular by Takeuchi 
\cite{takeuchi}, Masuoka and Wigner 
\cite{mw}, and M\"uller and Schneider 
\cite{ms} has shown, the following definition 
provides a reasonable generalisation of 
affine homogeneous spaces of algebraic groups
(see Section~\ref{kommutativ} below for some
discussion of the 
commutative case): 
\begin{dfn}\label{defqhs}
A \emph{quantum homogeneous space} 
is a right faithfully flat 
ring extension $B \subset A$ where 
$A=(A,\mu,\eta,\Delta,\varepsilon,S)$ 
is a Hopf algebra with bijective antipode 
$S$ over a field $k$ and $B$ is 
a right coideal subalgebra, $\Delta(B)
 \subset B \otimes A$. 
\end{dfn}

Our aim here
is to generalise
a theorem by Brown and Zhang \cite{bz} from Hopf
algebras to such subalgebras. For its
statement we adopt the following 
terminology 
(see Section~\ref{kommutativ}
for background information and motivation): 

\begin{dfn}\label{glaette}
Let $k$ be  field and $B$ be a (unital, associative) $k$-algebra. 
\begin{enumerate}
\item The \emph{dimension}  $\mathrm{dim} (B)$ of $B$ 
is its projective dimension 
in the category of finitely generated 
$B$-bimodules. $B$ is \emph{smooth} if $\mathrm{dim} (B)<\infty$.
\item A character $ \varepsilon : B \rightarrow k$ is
\emph{Cohen-Macaulay} if for the induced left $B$-module structure
on $k$ and some $d \ge 0$ one has $\mathrm{Ext}^n_B(k,B) = 0$ 
for $n \neq d$, and
\emph{Gorenstein} if in addition 
$\mathrm{Ext}^d_B(k,B) \simeq k$
as $k$-modules.
\end{enumerate} 
\end{dfn} 

Under these conditions we can deduce
a Poincar\'e-type duality in the Hochschild
(co)homology of $B$ as studied by Van den 
Bergh in \cite{vdb}: 

\begin{thm}\label{main}
If $B \subset A$ is a 
smooth quantum homogeneous space and
the restriction of 
$\varepsilon$ to $B$ is Cohen-Macaulay,
then there are isomorphisms
\begin{equation}\label{pd}
		  \mathrm{Ext}^n_{B^e}(B,\,\cdot\,) \simeq 
			\mathrm{Tor}^{B^e}_{\mathrm{dim}(B)-n}(\omega \otimes_B \cdot\,,B),\quad
			\omega := \mathrm{Ext}_{B^e}^{\mathrm{dim}(B)}
			(B,B^e)
\end{equation} 
of functors on the category of
$B$-bimodules that are right flat.
Here $\otimes:=\otimes_k$, $B^e:=B \otimes B^\mathrm{op}$,
we identify left and right $B^e$-modules
and $B$-bimodules, and 
the $B$-bimodule structure on
$\omega$ is induced by right
 multiplication in $B^e$. 
\end{thm}

If $B=A$ and $\varepsilon$ is
Gorenstein, then Brown and Zhang's
result also says that  
$\omega \simeq A_\sigma$ for 
some $\sigma \in \mathrm{Aut}(A)$ \cite{bz}, by
which we mean it is isomorphic to $A$ as left
module but the right action is given by 
$a \blacktriangleleft b:=a \sigma
(b)$. In particular, $\omega$ is an invertible bimodule with inverse
$\omega^{-1} \simeq A_{\sigma^{-1}}$, so the duality
(\ref{pd}) can be reversed
to
\begin{equation}\label{pdinv}
		  \mathrm{Tor}^{B^e}_n(\,\cdot\,,B) \simeq 
			\mathrm{Ext}_{B^e}^{\mathrm{dim}(B)-n}(B,\omega^{-1}
			\otimes_B \cdot\,),\quad
			\omega \otimes_B \omega^{-1}
			\simeq \omega^{-1} \otimes_B
			\omega \simeq B,
\end{equation} 
and this holds in fact on the category
of all $B$-bimodules (the flatness
assumption 
becomes obsolete), see \cite{vdb}. Then
the duality is not only of theoretical
interest but a valuable
tool when explicitly computing  
the Hochschild cohomology of $B$, see \cite{arabian}
for a concrete demonstration.

Algebraic geometry suggests that the
Gorenstein condition implies the invertibility of $\omega$
in greater generality: we will show in
Theorem~\ref{schmaerz} that $\omega$
carries in the Gorenstein case the
structure of a $(B,A)$-Hopf
bimodule. These are noncommutative
generalisations of the modules of
sections of homogeneous 
vector bundles, and 
$\mathrm{Ext}^{\mathrm{dim}
(B)}_B(k,B)$ reduces for commutative
rings to the typical fibre. So the Gorenstein
condition means here that we are dealing with a
line bundle whose module of sections is 
invertible.

We will recall that any quantum
homogeneous space can be written as 
\begin{equation}\label{edwin}
		  B =\{a \in A\,|\,(\pi \otimes
\mathrm{id}_A) \circ \Delta(a)=\pi (1)
\otimes a\},
\end{equation} 
where $\Delta$ is the coproduct in $A$ and
$\pi$ is the canonical projection 
onto $A/B^+A$, 
$B^+:=B \cap \mathrm{ker}\,
\varepsilon$, see Section~\ref{qhs}. 
This is a Hopf algebra map 
if and only if 
$A B^+=B^+ A$ (since
$B^+ A=S(A B^+)$ as observed by Koppinen,
see \cite{ms}, Lemma~1.4).
Our second main result applies to this case:
   
\begin{thm}\label{mainneu}
If $B \subset A$ is a smooth and
 Gorenstein quantum
 homogeneous space with $AB^+=B^+A$,
then $\mathrm{Ext}^{\mathrm{dim}
 (B)}_{B^e}(B,B^e)$ is
an invertible $B$-bimodule.
\end{thm}

The condition $AB^+=B^+A$ holds
trivially if $A$ is  
the commutative coordinate ring of an algebraic group
$G$. Then $B$ is the coordinate
ring $k[X]$ of an affine homogeneous space of $G$,
and $A/B^+A$ is the coordinate ring $k[H]$ 
of the isotropy group $H \subset G$ of 
$X \simeq H \setminus \!G$, see Section~\ref{kommutativ}. 
Important noncommutative examples with $AB^+=B^+A$
are quantisations of quotients $H \setminus\! G$ of
a Poisson group by a Poisson subgroup,
such as the standard quantisations
of the generalised flag manifolds studied
e.g.~in~\cite{chpr,kolb1,kolb2,kolb3,ulid,stokman}.

There are, however, plenty examples of 
quantum homogeneous spaces 
with $AB^+\neq B^+A$ such as the
nonstandard Podle\'s spheres
\cite{Po,ms} and more generally
quantisations of quotients
of Poisson groups by coisotropic
subgroups. We use the antipode of
$A/B^+A$ explicitly when
constructing $\omega^{-1}$ but 
we are not aware of a 
counterexample
to Theorem~\ref{mainneu} with the
assumption $AB^+=B^+A$ removed, and we
expect its conclusion holds for the
nonstandard Podle\'s spheres. Hence 
we ask: 
\begin{question}
Is $\omega$ invertible for all smooth
quantum homogeneous
spaces when $\varepsilon$ is Gorenstein?
\end{question}

\subsection{Application}
Our main motivation is to apply our
results to the paradigmatic example of a quantum homogeneous space 
which is Podle\'s' standard quantum sphere \cite{Po}. Here
$A$ is the quantised coordinate ring
$\mathbb{C}_q[SL(2)]$, and $A/B^+A
\simeq \mathbb{C}[z,z^{-1}]$. The quotient
$\pi$ 
deforms the map dual to the embedding of a 
maximal torus $T \simeq \mathbb{C}^*$ into 
$SL(2,\mathbb{C})$, so $B$
deforms  the coordinate ring of 
the coset space $T \setminus \!
SL(2,\mathbb{C})$ which is isomorphic to
the complexified 2-sphere given in $ \mathbb{C}^3$ by
$x^2+y^2+z^2=1$.   
We will prove that $B$ 
satisfies all the homological 
assumptions of Theorem~\ref{mainneu} and
compute $\omega$:
\begin{thm}\label{appli}
Let $q \in \mathbb{C}^*$ 
be not a root of unity and 
$A$ be the quantised
coordinate ring of $SL(2,\mathbb{C})$.
Then the standard 
Podle\'s quantum 2-sphere $B \subset A$ 
is smooth with
$ \mathrm{dim} (B)=2$, 
$ \varepsilon|_B$ 
is Gorenstein, and we have
$ \omega \simeq B_\sigma $,
where $\sigma$ is the restriction of
the square $S^2$ of the antipode of $A$ to $B$.
\end{thm}

This form of $\omega$ had to be expected 
from Dolgushev's results \cite{dolgushev}
in the setting of formal deformation
quantisations, and Hadfield's computations \cite{tom},
since $S^2|_B$ quantises the flow of the 
modular vector field of the quantised 
Poisson structure on the 2-sphere, and also
coincides with the modular automorphism of the Haar functional of 
$A$, see Section~\ref{stpodl} for
further details.

As we mentioned above, the standard quantum 2-sphere can 
be further deformed 
to quantum homogeneous spaces of $\mathbb{C}_q[SL(2)]$
where $A B^+ \neq B^+A$ \cite{Po,ms}. 
The Gorenstein condition is checked for 
these in the same way as for the standard 
sphere. It was shown in 
\cite{bavula} that their global dimension
is 2, but the methods used there 
seem not to allow us to answer
\begin{question}
Are the nonstandard 
Podle\'s spheres smooth? 
\end{question}       

\subsection{The case of coordinate rings}\label{kommutativ}  
For the reader's 
convenience we briefly recall here 
the geometric 
background of the theory in the case that
$B \subset A$ are coordinate rings 
of affine varieties over 
an algebraically closed field.

A Hopf
algebra structure on the coordinate ring 
$A=k[G]$ of an affine variety $G$  
corresponds directly 
to an algebraic group 
structure on $G$. 
Furthermore, a faithfully flat 
embedding $B=k[X] \subset A$
corresponds to a surjection 
$G \rightarrow X$ (\cite{matsumura2}, 
Theorem~7.3 on p.~48). Since
$ \Delta (B) \subset 
B \otimes A \simeq k[X \times G]$, 
$ \Delta $ defines 
an algebraic action 
$X \times G \rightarrow X$ of $G$ on $X$ for 
which the quotient map $G \rightarrow X$
is equivariant. Hence $X$ is indeed a 
homogeneous space of $G$, that is, the action is
transitive and $X \simeq H \setminus \! G$
for a closed subgroup $H \subset G$.

Recall next that a variety $X$ is smooth in
a point if and only if its local ring in the point has
finite global dimension which is then 
equal to $ \mathrm{dim}(X)$
(\cite{matsumura2}, 
Theorem~19.2 on p.~156). 
Since $ \mathrm{Ext} $ is compatible with 
localisations in the sense that 
for all maximal ideals 
$\mathfrak{m}$ in a 
commutative Noetherian ring $B$ and 
all finitely generated modules 
$M,N$ over $B$ one has
(\cite{weibel},~Proposition~3.3.10) 
\begin{equation}\label{lokalis}
(\mathrm{Ext}^\bullet_B(M,N))_\mathfrak{m} \simeq 
 \mathrm{Ext}^\bullet_B(M,N) \otimes_B B_\mathfrak{m}
\simeq
\mathrm{Ext}^\bullet_{B_\mathfrak{m}}(M_\mathfrak{m},
N_\mathfrak{m}),
\end{equation}  
$X$ is smooth in all points 
if and only if $ \mathrm{gl.dim}
(k[X])<\infty$.

One has in general 
$\mathrm{gl.dim}(B) 
\le \mathrm{dim} (B)$
(see Lemma~\ref{wudi} in Section~\ref{mpro}), 
so the smoothness from 
Definition~\ref{glaette} implies for 
$B=k[X]$ that $X$ is smooth in all points. It can happen that 
$ \mathrm{dim} (B)=\infty$ even when 
$ \mathrm{gl.dim}(B)=0 $ (consider 
e.g.~$B=\mathbb{C}$ over $k=\mathbb{Q}$), 
but for $k=\bar k$, $k[X]$-bimodules are the same
as modules over
$k[X] \otimes k[X] \simeq k[X \times X]$, and this has finite
global dimension if $k[X]$ has  
(\cite{hkr}, Theorem~2.1) and is 
Noetherian. Hence the finitely generated 
$k[X] \otimes k[X]$-module
$k[X]$ admits a finitely 
generated projective resolution of
finite length and
$ \mathrm{dim} (k[X])<\infty$. Thus  
smoothness as in Definition~\ref{glaette} is 
really equivalent to geometric smoothness of $X$.

For a classical homogeneous 
space the smoothness
condition in Theorem~\ref{main} 
becomes in fact 
void in characteristic zero: 
Corollary~\ref{moninervt} below tells that 
an affine homogeneous 
space $X \simeq H \setminus \! G$ 
is smooth if $G$ is so, and 
affine algebraic groups
are smooth in characteristic zero, see  
\cite{waterhouse} Sections~11.6 and 11.7.

Similarly, a smooth character of a coordinate
ring is Gorenstein since this is for these equivalent to the finiteness of the
injective dimension of the corresponding
local ring as a module over itself
(\cite{matsumura2}, Theorem~18.1 on p.~141
in combination with (\ref{lokalis})). In
the noncommutative case this equivalence
breaks down which results in 
various nonequivalent generalisations of
the Gorenstein and similarly the
Cohen-Macaulay condition. The ones from
Definition~\ref{glaette} are closest in
spirit to the notions of AS Gorenstein and
AS Cohen-Macaulay rings \cite{jz}
but still more naive
and just meant as a working terminology
to be used within this paper.

Lastly we remark that the coordinate ring of any smooth affine variety
satisfies the duality from Theorem~\ref{main}
with $\omega$
being the inverse of the module of 
top degree K\"ahler differentials (algebraic
differential forms), see e.g.~\cite{uli}. 

\subsection{Structure of the paper}
Theorems~\ref{main} and \ref{mainneu}
are proved in Section~\ref{teoria}.
Sections~\ref{hohohi}-\ref{qhs}
recall background material on Hochschild
(co)homology and quantum homogeneous
spaces, mainly from \cite{nu,vdb} and \cite{mw,ms,takeuchi}. 
Section~\ref{ginseng} extends the
description of the Hochschild cohomology
of a Hopf algebra $A$ as a derived functor
over $A$ rather than $A^e$ to quantum
homogeneous spaces. Using this 
we prove Theorem~\ref{main} in 
Section~\ref{mpro}.

In
Section~\ref{omestru} we give 
$\omega=\mathrm{Ext}^{\mathrm{dim}(B)}_{B^e}(B,B^e)$ for
smooth and Gorenstein quantum
homogeneous spaces $B \subset A$  
the structure of a $(B,A)$-Hopf
bimodule and deduce that it is as a
left $B$-module isomorphic to
$$
		  \{a \in A\,|\,(\pi
		  \otimes \mathrm{id}_A) \circ
		  \Delta (a) = g \otimes a\}
$$ 
for some group-like element $g \in
 C=A/B^+A$. Using this we construct in
 Section~\ref{hoga} under the assumption $A
 B^+=B^+A$ a $B$-bimodule $\bar \omega$ with $\bar \omega
\otimes_B \omega  \simeq B_\sigma$ for
some algebra endomorphism $\sigma$ of
$B$. Section~\ref{chrks} discusses a
generalisation of the 
transitive action of a group $G$ on $X=H \setminus\!G$ to
characters on quantum homogeneous spaces. This is used to show in
Section~\ref{director} that  
$\sigma$ is an automorphism which 
implies Theorem~\ref{mainneu}.

A short Section~\ref{criterium} 
contains a criterion to prove the
smoothness of some quantum homogeneous spaces
which is applied later in the proof of Theorem~\ref{appli}, and Section~\ref{countex} 
gives three examples of quantum homogeneous 
spaces that illustrate certain aspects of the
general theory developed so far.

Section~\ref{applikazie} is devoted to
the Podle\'s sphere and the proof of Theorem~\ref{appli}. 

\subsection*{Acknowledgements}
I am deeply indebted to the referee who
pointed out that an originally submitted
version of the paper was
incorrect and also has made other
suggestions for improving the paper. 
Equally warm thanks go to Ken Brown, 
Tomasz Brzezi\'nski and Stefan Kolb,
to my EPSRC
fellowship EP/E/043267/1 and the
Polish Government Grant N201 1770 33. 

\section{Theory}\label{teoria} 
\subsection{Hochschild (co)homology}\label{hohohi}
Let $k$ be a field. For a $k$-algebra $B$, we denote by 
$B^\mathrm{op}$ 
the opposite algebra (same vector space, opposite 
multiplication) and by $B^e:=B \otimes B^\mathrm{op}$
the enveloping algebra of $B$ (here and
in the rest of the paper, an unadorned
$\otimes$ denotes the tensor product
over $k$). The tensor
flip $ \tau (a \otimes b):=b \otimes a$ 
defines a canonical isomorphism 
$(B^e)^\mathrm{op}\simeq B^e$ and hence 
identifies left and
right $B^e$-modules, and these are also the same
as $B$-bimodules (with symmetric action of $k$).
For any such bimodule $M$,
the Hochschild (co)homology 
of $B$ with coefficients in $M$ is 
$$
		  H_\bullet(B,M):=\mathrm{Tor}_\bullet^{B^e}(M,B),\quad
		  H^\bullet(B,M):=\mathrm{Ext}^\bullet_{B^e}(B,M),
$$
where $B$ is considered as a $B$-bimodule using
multiplication in $B$.

The bar
resolution of $B$ yields canonical
(co)chain complexes computing 
$H_\bullet(B,M)$ and $H^\bullet(B,M)$. For 
cohomology, this cochain complex is 
$$
		  C^\bullet(B,M):=\mathrm{Hom}_k(B^{\otimes
		  \bullet},M)
$$ 
with the coboundary
operator $\bb : C^n(B,M)
\rightarrow C^{n+1}(B,M)$ given by
\begin{align}\label{sputzens}
\bb \varphi (b^1,\ldots,b^{n+1}) 
&= b^1 \varphi(b^2,\ldots,b^{n+1})
\nonumber\\ 
 &\quad+\sum_{i=1}^{n} (-1)^i 
		  \varphi (b^1,\ldots,b^ib^{i+1},\ldots,b^{n+1})\\ 
 &\quad+(-1)^{n+1} \varphi (b^1,\ldots,b^n)
		  b^{n+1}.
		  \nonumber
\end{align}

For further information see 
e.g.~\cite{ce,loday,weibel}. 
 
\subsection{Van den Bergh's theorem}
The following theorem was proven by Van den Bergh 
in \cite{vdb}. To be precise, Van
den Bergh considered the case in which
the bimodule $\omega$ is invertible. For
the sake of clarity we include the
sketch of a proof not using this
assumption, see \cite{nu} for details. 
\begin{thm}\label{michel}
Let $B$ be a smooth algebra and assume
there exists $d \ge 0$ such that  
$H^n(B,B^e)=0$ for $n \neq d$.
Then $d=\mathrm{dim}(B)$ 
and there is for all $n \ge 0$ and for
every right $B$-flat $B$-bimodule $M$ a
 canonical isomorphism
\begin{equation}\label{duelliti}
		  H_n(B,\omega \otimes_B M) \simeq 
		  H^{d-n}(B,M),\quad
		  \omega:=H^d(B,B^e),
\end{equation}
where the
bimodule structure of $\omega$ is 
induced by right multiplication in $B^e$.  		  
\end{thm}
\begin{proof}
The assumption that $B$ is smooth means
that the $B^e$-module $B$ admits a resolution 
$P_\bullet$ of finite length consisting
of finitely generated
projective $B^e$-modules. Using 
$H^n(B,B^e)=0$ for $n \neq d$ and 
Schanuel's lemma one can 
assume without loss of generality
(see the proof of Theorem~23 in
 \cite{nu} for the detailed argument) that
this resolution has length $d$, and then
$P_{d-\bullet}^*:=\mathrm{Hom}_{B^e}(P_{d-\bullet},B^e)$
is a finitely generated projective
 resolution of $\omega$. Therefore we have canonical isomorphisms 
\begin{equation}
		  \mathrm{Hom}_{B^e}(P_\bullet,M)
		  \simeq
		  P_\bullet^* \otimes_{B^e} M
		  \simeq
		  (P_\bullet^* \otimes_B M)
		  \otimes_{B^e} B.
\end{equation} 
As a right $B^e$-projective
 module, $P^*_\bullet$ is right
 $B$-flat, so $P^*_\bullet
 \otimes_B M$ is a resolution of $\omega
\otimes_B M$. Furthermore, one easily
 convinces oneself that $P^*_\bullet
 \otimes_B M$ is $B^e$-flat if
$M$ is right $B$-flat (taking into
 account that $P_\bullet^*$ is finitely
 generated projective over $B^e$). Hence taking   
homology in the above equation gives
$$
		  \mathrm{Ext}^n_{B^e}(B,M) \simeq
		  \mathrm{Tor}_{d-n}^{B^e}(\omega
 \otimes_B M,B)
$$
as claimed.
\end{proof}

The point of our main result Theorem~\ref{main} is that the
condition about $H^\bullet(B,B^e)$ can be replaced
for quantum homogeneous spaces by the Cohen-Macaulay
condition which is easier to check for concrete
examples as we shall see below (it boils down to
constructing resolutions of the $B$-module 
$k$ rather than of the $B$-bimodule $B$). In the
commutative case, Van den Bergh's condition is 
a global one concerned with the behaviour 
of the embedding of the 
corresponding space $X$ as the diagonal 
into $X \times X$, while the Cohen-Macaulay condition 
in Theorem~\ref{main} is local in nature, dealing 
only with the local properties of $X$ around the point
corresponding to $ \varepsilon $.  

\subsection{Quantum homogeneous
  spaces}\label{qhs}  
We will freely use standard
conventions and notations from Hopf
algebra theory. In particular, we denote by
$\Delta,\varepsilon,S$ the coproduct,
counit and antipode of a co-  or Hopf
algebra and use Sweedler's notation 
$\Delta (a)=a_{(1)} \otimes
a_{(2)}$ for coproducts and 
$m \mapsto m_{(-1)} \otimes m_{(0)}$
and $n \mapsto n_{(0)} \otimes n_{(1)}$
for left and right coactions, see e.g.~\cite{chef,sweedler}.

We recall in this section from
\cite{mw,ms,takeuchi} various characterisations of
the right faithful flatness of a Hopf algebra
$A$ over a right
coideal subalgebra $B$ that we use later. Some of them
are given in terms of the
left coaction
$$
		  A \rightarrow C \otimes A,\quad
		  a \mapsto a_{(-1)} \otimes
		  a_{(0)}:=
		  \pi (a_{(1)}) \otimes a_{(2)},
$$
where we write as in the introduction
$$
		   \pi : A \rightarrow C:=A/B^+A,\quad a \mapsto 
			\pi (a) := a \mbox{ mod } B^+A,\quad   
		  B^+:=B \cap \mathrm{ker}\,\varepsilon.
$$
Yet others involve the
categories $\M^C$ and ${}_B \M^A$ of right $C$-comodules
and of 
$(B,A)$-Hopf modules, meaning
left $B$-modules and 
right $A$-comodules $M$ for which the
coaction $M \rightarrow M \otimes A$ is
$B$-linear if $B$ acts on 
$M \otimes A$ via 
$$
		  b(m \otimes a):=b_{(1)}m \otimes
		  b_{(2)}a\quad
		  b \in B,m \in M,a \in A.
$$
There are two functors relating these
two categories. The first one is
\begin{equation}\label{drueckt}
		  {}_B \M^A \rightarrow \M^C,\quad
		  M \mapsto M/B^+M,
\end{equation} 
where the $C$-coaction on $M/B^+M$
is induced by the
$A$-coaction on $M$, and the second one is the
cotensor product
\begin{equation}\label{cott}
		  \M^C \rightarrow {}_B \M^A,\quad
			N \mapsto N \Box_C A
\end{equation} 
given for $N \in \M^C$ with coaction 
$N \rightarrow N \otimes C$,
$n \mapsto
n_{(0)} \otimes n_{(1)}$ 
by
$$
		  N \Box_C
		  A:=\{\sum_i n^i \otimes
		  a^i \in N \otimes A\,|\,
		  \sum_i n^i_{(0)} \otimes
		  n^i_{(1)} \otimes a^i=
		  \sum_i n^i \otimes a^i_{(-1)} \otimes a^i_{(0)}\}
$$
on which the $B$-action and $A$-coaction
are given by (co)multiplication in $A$.

The following is \cite{mw}, Theorem~2.1 and 
\cite{ms}, Theorem~1.2 and Remark~1.3: 
\begin{thm}\label{dirsum}
Let $A$ be a Hopf algebra with bijective
 antipode and
$B \subset A$ be a right coideal
subalgebra. Then the following are
 equivalent:
 \begin{enumerate}
\item $A$ is faithfully flat as a right
		module over $B$.
\item $A$ is projective as a right $B$-module 
and there exists 
$B^\perp \subset A$ such that 
$A=B \oplus B^\perp$ as right $B$-module.
\item The functors (\ref{drueckt}) and
		(\ref{cott}) are (quasi)inverse
		equivalences.
\item $A$ is left $C:=A/B^+A$-coflat and we have
$$
		  B=k \Box_C A=
		  \{b \in A\,|\,\pi(b_{(1)}) \otimes b_{(2)}=
		  \pi(1) \otimes b\}.
$$
\end{enumerate}  
\end{thm}   

If $A B^+=B^+ A$, then Remark~1.3 in
\cite{ms} also tells that $A$ is faithfully flat as a
left module if it is faithfully flat as a right
module. 
\begin{question}
Is this true in general? 
\end{question}

By Theorem~\ref{dirsum} (4) a quantum homogeneous space 
can be recovered from $ \pi : A \rightarrow C$ 
as $k \Box_C A$.
Many examples are in fact defined in this way 
starting with $ \pi $.  
This works in 
particular when $C$ is 
cosemisimple (equals the direct sum of its
simple subcoalgebras), see 
e.g.~\cite{ms}, Corollary~1.5: 

\begin{cor}\label{schneller}
Let $A$ be a Hopf algebra with bijective 
antipode, $ \pi : A \rightarrow C$ 
be a coalgebra and right 
$A$-module quotient, and assume that
$C$ is cosemisimple. Then
$B:=k \Box_C A \subset A$ is a quantum 
homogeneous space and $C \simeq A/B^+ A$.
\end{cor}

For $C=k[H]$
cosemisimplicity means that 
$H$ is reductive, so 
a quotient $H \setminus \! G$ of an 
algebraic group
$G$ by a reductive subgroup $H$ is 
affine with
coordinate ring $B$ isomorphic to 
the ring of 
$H$-invariant regular functions on $G$ 
(this is essentially the classical 
Matsushima-Onishchik theorem). 
A
$(B,A)$-Hopf module $M \in {}_B \M^A$ is
here isomorphic
to the module of sections of the
$G$-homogeneous vector bundle with typical fibre
$M/B^+M$.

Later we
will also use categories that we
denote by $\M^C_{B,\tau}$ and by ${}_B
\M^A_{B,\tau}$, where $\tau : B
\rightarrow A$ is an algebra map. By the first we shall mean the
category of right $C$-comodules and
right $B$-modules $N$ that satisfy 
\begin{equation}\label{britzel}
		  (nb)_{(0)} \otimes
			(nb)_{(1)}=n_{(0)}b_{(1)} \otimes
		  n_{(1)}\tau(b_{(2)}), 
\end{equation} 
where we use in
the second tensor component on the right
hand side the right $A$-action on
$A/B^+A$. Similarly, objects in
${}_B \M^A_{B,\tau}$ are objects in 
${}_B \M^A$ with an additional right
$B$-action that commutes with the left
one and satisfies (\ref{britzel}), now
being an equation in $M \otimes A$.  
Clearly, the equivalence ${}_B \M^A
\simeq \M^C$ that holds in the
faithfully flat case also induces
${}_B \M^A_{B,\tau}
\simeq \M^C_{B,\tau}$.

\subsection{$H^\bullet(B,M)$ and 
$\mathrm{Ext}^\bullet_B(k,\mathrm{ad}(M))$}\label{ginseng}
Here we remark that the description of the Hochschild
(co)homology of a Hopf algebra $A$ 
used in \cite{bz} works almost as well for quantum
homogeneous spaces $B \subset A$.
The proof is the same as for $B=A$ 
\cite{ft,gk}, we recall it only for the convenience of the
reader:

\begin{lemma}\label{gklem}
Let $A$ be a Hopf algebra, $B \subset A$ be a right
coideal subalgebra, and $M$ be a $B$-$A$-bimodule.
Consider $k$ as left $B$-module with action given by 
the counit $ \varepsilon $ of $A$, and let $\mathrm{ad}(M)$   
be the left $B$-module which is $M$ as vector space with
 left action given by the adjoint action
$\mathrm{ad}(b)m:=b_{(1)} m S(b_{(2)})$.
Then there is a vector space isomorphism 
$H^\bullet(B,M) \simeq 
\mathrm{Ext}^\bullet_B(k,\mathrm{ad}(M))$.  
\end{lemma}
\begin{proof}
Compute 
$\mathrm{Ext}^\bullet_B(k,\mathrm{ad}(M))$
using the free resolution
$$
		  \ldots \rightarrow B^{\otimes 3} \rightarrow 
		  B^{\otimes 2} \rightarrow B
$$ 
of the $B$-module $k$ whose
boundary map is given by 
$$
		  b^0 \otimes \cdots \otimes b^n \mapsto 
		  \sum_{i=0}^{n-1} (-1)^i 
		  b^0 \otimes \cdots \otimes b^ib^{i+1} \otimes 
		  \cdots \otimes b^n
		  +(-1)^n b^0 \otimes \cdots \otimes
		  b^{n-1} \varepsilon (b^n).
$$  
After identifying $B$-linear maps 
$B^{\otimes n+1} \rightarrow M$ with 
$k$-linear maps $B^{\otimes n} \rightarrow M$ (fill the
 zeroth tensor component with $1 \in B$),
this realises $\mathrm{Ext}^\bullet_B(k,\mathrm{ad}(M))$
as the cohomology of the cochain complex 
which as a vector space is
$$
		  C^\bullet(B,M)=
		  \mathrm{Hom}_k(B^{\otimes \bullet},M),
$$ 
the standard Hochschild cochain
complex, but whose coboundary map is
\begin{align}
& \di \varphi (b^1,\ldots,b^{n+1}) \nonumber\\ 
=\quad& \mathrm{ad} (b^1) \varphi(b^2,\ldots,b^{n+1})
		 +\sum_{i=1}^{n} (-1)^i 
		  \varphi (b^1,\ldots,b^ib^{i+1},\ldots,b^{n+1}) \nonumber\\ 
&+(-1)^{n+1} \varphi (b^1,\ldots,b^n)\varepsilon
		  (b^{n+1}).
		  \nonumber
\end{align}
Now consider the $k$-linear isomorphism 
$$
		  \xi : C^\bullet(B,M) \rightarrow C^\bullet(B,M),\quad
		  (\xi (\varphi ))(b^1,\ldots,b^n):=
		  \varphi (b^1_{(1)},\ldots,b^n_{(1)})
		  b^1_{(2)} \cdots b^n_{(2)}
$$
whose inverse is given by 
$$
		  (\xi^{-1} (\varphi ))(b^1,\ldots,b^n):=
		  \varphi (b^1_{(1)},\ldots,b^n_{(1)})
		  S(b^1_{(2)} \cdots b^n_{(2)}).
$$
Then
$\bb \circ \xi = \xi \circ \di$, where $\bb$ is the standard Hochschild
 coboundary operator (\ref{sputzens}),
 so $(C^\bullet(B,M),\di) \simeq
 (C^\bullet(B,M),\bb)$ as cochain complexes.
\end{proof}

One can apply Theorem~VIII.3.1 from \cite{ce}
to the map $B \rightarrow B \otimes A^\mathrm{op}$,
$b \mapsto b_{(1)} \otimes S(b_{(2)})$
to show
$\mathrm{Ext}^\bullet_{B \otimes A^\mathrm{op}} 
(A,M) \simeq \mathrm{Ext}^\bullet_B(k,\mathrm{ad}(M))$.
When $A$ is flat over $B$, then the same theorem
applied to the obvious embedding of
$B \otimes B^\mathrm{op}$ into 
$B \otimes A^\mathrm{op}$ also implies
$\mathrm{Ext}^\bullet_{B \otimes A^\mathrm{op}} 
(A,M) \simeq H^\bullet(B,M)$ and hence the above
lemma. We included the above proof since
it does not require flatness. On the other hand, this 
seems to be a rather weak condition. It is always
satisfied in the commutative case \cite{mw},
note also the recent results of Skryabin~\cite{skryabin}. 
For a counterexample see \cite{pete},
Corollary~2.8
and Remark~2.9.
  
\subsection{The proof of Theorem~\ref{main}}\label{mpro}  
To get Theorem~\ref{main} we only have to consider the special
case $M=B \otimes A$ of
Lemma~\ref{gklem} in more detail. We
first recall:
\begin{lemma}\label{blemma}
Let $R,S$ be rings, $L$ be an
 $R$-module, $M$ be an $R$-$S$-bimodule
 and $N$ be an $S$-module. Then the
 canonical map 
$$\mathrm{Ext}^n_R(L,M)
 \otimes_S N \rightarrow
 \mathrm{Ext}^n_R(L,M \otimes_S N)
$$ is bijective if  $N$ is
 flat and $L$ admits a finitely
 generated projective resolution. 
\end{lemma}
\begin{proof}
Fix a finitely generated projective
resolution $P_\bullet \rightarrow
 L$. Then one has 
$\mathrm{Hom}_R(P_\bullet,M) \otimes_S N
 \simeq
\mathrm{Hom}_R(P_\bullet,M \otimes_S N)$,
see e.g.~\cite{bourbaki}, 
Proposition~8.b)
on p.~16. Now pass to cohomology
taking into account that $N$ is flat
(see e.g.~[ibid.], Corollary~2 
on p.~74).
\end{proof}

This will be used with $R=M=B,S=L=k$ and
$N=A$. For the assumption on $L=k$ we
recall from \cite{ce}: 
\begin{lemma}\label{wudi}
If $B$ is an algebra over a field $k$ and
$P_\bullet \rightarrow B$ is a 
(finitely generated) projective resolution of $B^e$-modules,
then $P_\bullet \otimes_B L$ is for any
left $B$-module a (finitely generated)
 projective resolution of
 $B$-modules. In particular, one has for
 any algebra
$\mathrm{gl.dim} (B) \le \mathrm{dim} (B) $.
\end{lemma}
\begin{proof}
The complex
$\ldots \rightarrow P_d \rightarrow \ldots
 \rightarrow P_0 \rightarrow B
 \rightarrow 0$ is a
 flat resolution of
the right $B$-module $0$. Therefore,
$H_\bullet(P \otimes_B L) \simeq \mathrm{Tor}^B_\bullet(0,L)=0$, so 
 $P_\bullet \otimes_B L$ is and it
 consists of (finitely generated) projective
 left $B$-modules. 
\end{proof}

Secondly, we need the following direct generalisation
of the case $B=A$:

\begin{lemma}\label{schgr}
Let $A$ be a Hopf algebra and $B \subset A$ be a right
coideal subalgebra. Then the $B$-$B
 \otimes A^\mathrm{op}$-bimodule $B \otimes A$ with 
actions
$$
		  \mathrm{ad}(x)(b \otimes a)(y \otimes z):=
		  x_{(1)}by \otimes zaS(x_{(2)})
$$ 
is isomorphic to the $B$-$B \otimes
 A^\mathrm{op}$-bimodule $B \otimes A$ with actions
$$
			x(b
			\otimes a) \triangleleft (y
			\otimes z):=
			xby_{(1)} \otimes zaS^2(y_{(2)}).
$$
\end{lemma}
\begin{proof}
The isomorphism is given explicitly by
$$
		  \rho : B \otimes A \rightarrow B \otimes A,\quad
		  b \otimes a \mapsto b_{(1)} \otimes a S^2(b_{(2)}).
$$
Its inverse is given by
$$
		  \rho^{-1} : b \otimes a \mapsto b_{(1)} \otimes a S(b_{(2)}), 
$$ 
and it follows straightforwardly from 
the Hopf algebra axioms that
\begin{equation*}
		  \rho (x_{(1)}by \otimes zaS(x_{(2)}))=
			x \rho (b
			\otimes a) \triangleleft (y
			\otimes z).\qedhere
\end{equation*}  
\end{proof}

Combining the lemmata gives:

\begin{thm}\label{gklem2}
Let $B \subset A$ be a right coideal subalgebra
and consider $B
 \otimes A$ as a
$B \otimes A^\mathrm{op}$-bimodule via
multiplication in $B \otimes
 A^{\mathrm{op}}$.
If the left $B$-module
$k$ admits a finitely generated
projective resolution, then there is an
 isomorphism
$$
		  H^\bullet(B,B \otimes A) \simeq
		  \mathrm{Ext}^\bullet_B(k,B)
		  \otimes A
$$
 of right $B \otimes
 A^\mathrm{op}$-modules, where $\mathrm{Ext}^\bullet_B(k,B)
\otimes A$ is a $B \otimes
 A^\mathrm{op}$-module via
$$
		  ([\varphi] \otimes a)(x \otimes y)
		  :=[\varphi]x_{(1)} \otimes yaS^2(x_{(2)}),\quad
		  x \in B,[\varphi] \in \mathrm{Ext}^\bullet_B(k,B),a,y \in A
$$ 
with the right $B$-action on $\mathrm{Ext}^\bullet_B(k,B)$
induced by right multiplication in $B$.
\end{thm}
\begin{proof}
Apply Lemma~\ref{gklem} with 
$M=B \otimes A$. The cochain complexes and
 the isomorphisms $\xi,\xi^{-1}$
defined in its proof are clearly right
 $B \otimes A^\mathrm{op}$-linear in this case, so the lemma
 gives a right $B \otimes
 A^\mathrm{op}$-module isomorphism
\begin{equation}\label{koc}
		  H^\bullet(B,B \otimes A) \simeq 
		  \mathrm{Ext}^\bullet_B(k,\mathrm{ad}(B
		  \otimes A)),
\end{equation} 
where the right
$B \otimes A^\mathrm{op}$-action
on $\mathrm{Ext}^\bullet_B(k,\mathrm{ad}(B
\otimes A)$ is induced by right
 multiplication in $B
 \otimes A^\mathrm{op}$ (which commutes with the
 left $B$-action given by $\mathrm{ad}$).
Now apply  the Lemmata~\ref{schgr},~\ref{wudi} and~\ref{blemma} to get the isomorphisms
\begin{equation}\label{bratsche}
		  \mathrm{Ext}_B^\bullet(k,\mathrm{ad}(B \otimes A))
		  \simeq 
		  \mathrm{Ext}_B^\bullet(k,B \otimes A) \simeq
		  \mathrm{Ext}_B^\bullet(k,B) \otimes A.
\end{equation} 
Composing these isomorphisms
with (\ref{koc}) yields the claim.
\end{proof}

Theorem~\ref{main} is an easy consequence:

\begin{proof}[Proof of Theorem~\ref{main}]
Theorem~\ref{dirsum} gives
a $B$-trimodule decomposition  
$$
		  B \otimes A \simeq B \otimes (B \oplus
		  B^\perp)  \simeq B^e \oplus 
		  (B \otimes B^\perp),
$$
so we also have
$
		  H^n(B,B \otimes A) \simeq 
		  H^n(B,B^e) \oplus 
		  H^n(B,B \otimes B^\perp)
$ 
as right $B$-modules. 
Theorem~\ref{gklem2} and the
 Cohen-Macaulay assumption imply that 
$H^n(B,B^e)=0$ for $n \neq \mathrm{dim} (B)$, so
 Theorem~\ref{main} follows from Theorem~\ref{michel}.
\end{proof}

\subsection{$\omega$ as a Hopf bimodule}\label{omestru} 
The key step towards
Theorem~\ref{mainneu} is to turn
$\omega$ into an object in ${}_B \M^A_{B,S^2}$.
Recall that any right $A$-comodule $N$
is via
\begin{equation}\label{hit}
		  X.n:=n_{(0)}X(n_{(1)}),\quad X \in
		  A^\circ,n \in N
\end{equation}  
a left module over the Hopf algebra
$A^\circ$ of
linear functionals on $A$ that vanish
on an ideal of finite
codimension, see e.g.~\cite{sweedler}
for background. The $A^\circ$-modules 
of this form are traditionally called
rational.
We define
now an $A^\circ$-action on $C^\bullet(B,B \otimes A)$ that restricts to $C^\bullet(B,B^e)$
and commutes with the coboundary
operator $\bb$ and therefore induces an $A^\circ$-action
on $\omega$. While 
$C^\bullet(B,B^e)$ will not be rational
in general we will prove afterwards
that $\omega$ is.

In the definition of the searched for
$A^\circ$-action on $\varphi \in
C^n(B,B \otimes A)$ we denote the canonical  
$A^\circ \otimes A^\circ$-action on $B
\otimes A$ by
$$
		  (X \otimes Y) \triangleright (x \otimes y)= 
		  X.x \otimes Y.y,\quad X,Y \in
		  A^\circ,x \in B,y \in A,
$$ 
where the actions of $X,Y$ result as in
(\ref{hit}) from the 
$A$-coactions given by
the coproduct. This gets mixed with an action on the
arguments of $\varphi$:
\begin{align}\label{ahmad} 
 (X\varphi)(b^1,\ldots,b^n)
&:= (S^2(X_{(n+2)}) \otimes X_{(1)})
		  \triangleright \\
&\qquad \varphi (S(X_{(n+1)}).b^1,\ldots,
		  S(X_{(2)}).b^n)).\nonumber
\end{align}
It follows from the Hopf algebra axioms
that this defines a left
$A^\circ$-action, and in this way  
 $C^\bullet(B,B \otimes A)$ becomes a cochain
 complex of $A^\circ$-modules:
\begin{lemma}
One has $\bb(X \varphi)=X(\bb
\varphi)$ for all $X \in A^\circ,\varphi
 \in C^\bullet(B,B \otimes A)$.
\end{lemma}
\begin{proof}
This is checked using 
that we have for $m \in B
\otimes A,
		  b,c \in B,X \in
		  A^\circ$
\begin{align}
 		  &X.(bc) =(X_{(1)}.b)(X_{(2)}.c),\nonumber\\ 
 & (X \otimes 1) \triangleright (bm)=
		  (X_{(1)}.b)((X_{(2)} \otimes 1)
 \triangleright m),\nonumber\\ 
& (1 \otimes X) \triangleright (mb)=
		  ((1 \otimes X_{(1)})
		  \triangleright
		  m)(X_{(2)}.b).\nonumber
\end{align} 
 We demonstrate the claim in
degree $n=1$, the general case is
analogous: 
\begin{align}
&\quad (X(\bb \varphi))(b,c) \nonumber\\ 
&= (S^2(X_{(4)}) \otimes X_{(1)}) \triangleright
		  (\bb\varphi(S(X_{(3)}).b,S(X_{(2)}) . c))
		  \nonumber\\ 
&= (S^2(X_{(4)}) \otimes X_{(1)}) \triangleright
		  ((S(X_{(3)}) . b)
		  \varphi(S(X_{(2)}) . c) \nonumber\\ 
&\quad -(\varphi(S(X_{(3)}) .b)(S(X_{(2)}) . c))
		  +\varphi(S(X_{(3)}) .b)(S(X_{(2)}) . c))
		  \nonumber\\
&= (S^2(X_{(4)}) \otimes X_{(1)}) \triangleright
		  ((S(X_{(3)}) . b)\varphi(S(X_{(2)}) . c))\nonumber\\ 
&\quad-(S^2(X_{(4)}) \otimes X_{(1)}) \triangleright 
		  (\varphi((S(X_{(3)}).b)(S(X_{(2)}) . c)))
		  +\nonumber\\ 
&\quad(S^2(X_{(4)}) \otimes X_{(1)}) \triangleright (\varphi(S(X_{(3)}) .
		  b)(S(X_{(2)}) . c))
		  \nonumber\\
&=  (S^2(X_{(4)})S(X_{(3)}) . b)((S^2(X_{(5)}) \otimes X_{(1)}) \triangleright
		  (\varphi(S(X_{(2)}) . c)))\nonumber\\ 
&\quad-(S^2(X_{(3)}) \otimes X_{(1)}) \triangleright 
		  (\varphi((S(X_{(2)})_{(1)}.b)(S(X_{(2)})_{(2)} . c)))
		  +\nonumber\\ 
&\quad((S^2(X_{(5)}) \otimes X_{(1)}) \triangleright (\varphi(S(X_{(4)}) .
		  b)))(X_{(2)}S(X_{(3)}) . c)
		  \nonumber\\
&= b ((S^2(X_{(3)}) \otimes X_{(1)}) \triangleright
		  \varphi(S(X_{(2)}) . c))
		  \nonumber\\ 
&\quad-(S^2(X_{(3)}) \otimes X_{(1)}) \triangleright (\varphi(S(X_{(2)}) .
		  (bc)))+\nonumber\\ 
&\quad((S^2(X_{(3)}) \otimes X_{(1)})\triangleright 
		  (\varphi(S(X_{(2)}).b)))c
		  \nonumber\\
&=  \bb(X \varphi)(b,c).\nonumber\qedhere
\end{align}
\end{proof}

Furthermore, we obviously have:
\begin{lemma}
For any right coideal subalgebra $B
 \subset A$, the canonical map
 $C^\bullet(B,B^e) \subset C^\bullet(B,B
 \otimes A)$ is an embedding of  complexes of $A^\circ$-modules.
\end{lemma}

Thus we obtain an $A^\circ$-action on
$H^\bullet(B,B^e)$ and the canonical map to 
$H^\bullet(B,B \otimes A)$ is
$A^\circ$-linear. Our final aim is to
prove that these two $A^\circ$-modules
are for a smooth and Gorenstein quantum
homogeneous space rational, and that we indeed have
$\omega \in {}_B \M^A_{B,S^2}$.

\begin{lemma}\label{geschrei}
Let $B \subset A$
be a smooth right coideal subalgebra
and assume $\varepsilon|_B$ is Gorenstein. Let $ \chi : B \rightarrow k$
be the character defined by the
right $B$-action on  
$\mathrm{Ext}^{\mathrm{dim}(B)}_B(k,B) \simeq k$
and define
the $k$-algebra homomorphism
\begin{equation}\label{emma}
		  \sigma : B \rightarrow A,\quad \sigma (x) := S^2(\chi(x_{(1)})x_{(2)}).
\end{equation} 
Then there are isomorphisms of
 $A$-$B$-bimodules
and $A^\circ$-modules
$$
		  H^n(B,B \otimes A) \simeq \left\{
 \begin{array}{ll}
0 \quad & n \neq \mathrm{dim}(B),\\
A_\sigma \quad & n=\mathrm{dim}(B),
 \end{array}\right.
$$
where $A^\circ$ acts via the canonical
 action $X.a:=a_{(1)}X(a_{(2)})$ on $A_\sigma$.
\end{lemma} 
\begin{proof}
The claim about $A$-$B$-bimodules is a straightforward application of
Theorem~\ref{gklem2}.    
One then has to transport the
 $A^\circ$-action on $C^\bullet(B,B
 \otimes A)$ though the used
 isomorphisms: conjugating it by $\xi$
from Lemma~\ref{gklem}
gives an $A^\circ$-action on the cochain complex
$(C^\bullet(B,B \otimes A),\di)$ that is
given by
$$
		  (X \blacktriangleright \varphi)(b^1,\ldots,b^n):=
		  (S^2(X_{(2)}) \otimes X_{(1)}) \triangleright  
		  \varphi (b^1,\ldots,b^n),
$$
 so this action is entirely induced from
 an action on the coefficient
 bimodule. Conjugating this action with
$\rho$ from Lemma~\ref{schgr}
gives the action
$$
		  (X.\varphi)(b^1,\ldots,b^n):=
		  (1 \otimes X) \triangleright 
		  \varphi (b^1,\ldots,b^n),
$$
that is, in the identifications
 (\ref{koc}) and
 (\ref{bratsche}) in the proof of Theorem~\ref{gklem2} 
the original $A^\circ$-action on
 $H^{\mathrm{dim} (B)}(B,B \otimes A)$ induced by
 (\ref{ahmad}) is
 transformed into the one on
$\mathrm{Ext}^{\mathrm{dim}(B)} _B(k,B) \otimes A$ where
 $A^\circ$ acts simply on the second
 tensor component $A$ in the canonical way. 
\end{proof}

In particular, $H^{\mathrm{dim} (B)}(B,B
\otimes A)$ is a
rational $A^\circ$-module, and
hence so is any $A^\circ$-submodule (\cite{sweedler},
Theorem~2.1.3.a). Furthermore,
$A_\sigma$ and hence any 
$B$-subbimodule and
$A$-subcomodule is an object
in ${}_B \M^A_{B,S^2}$. This gives:

\begin{cor}\label{goedel}
If $B \subset A$ is a smooth quantum
 homogeneous space and 
$\varepsilon|_B$ is Gorenstein, then
$\omega=H^{\mathrm{dim}(B)}(B,B^e)$
becomes through the embedding 
$$\omega=H^{\mathrm{dim} (B)}(B,B^e) 
 \rightarrow H^{\mathrm{dim} (B)}(B,B
 \otimes A) \simeq A_\sigma
$$ 
an object in ${}_B \M^A_{B,S^2}$.
\end{cor}

This allows us to describe
$\omega$ finally as follows using the
canonical projection $ \pi : A \rightarrow
C=A/B^+A$: 

\begin{thm}\label{schmaerz}
Let $B \subset A$ be a smooth quantum
 homogeneous for which
 $\varepsilon|_B$ is Gorenstein, and let
 $\chi$ be the character on $B$ defined
 by its action on
 $\mathrm{Ext}^{\mathrm{dim} (B)}_B(k,B)
 \simeq k$.
Then there exists a group-like $g \in
 C=A/B^+A$ with
$$
		  \omega \simeq \{a \in A_\sigma\,|\,\pi
 (a_{(1)}) \otimes a_{(2)}=g \otimes
 a\},\quad
 \sigma (b)=\chi (b_{(1)})S^2(b_{(2)})
$$ 
as an object of ${}_B \M^A_{B,S^2}$, and we
 have for all $b \in B$
\begin{equation}\label{knights}
					g \sigma (b)=\chi (b)g,
\end{equation} 
where $g \sigma (b)$ is defined using
 the right $A$-action on $C=A/B^+A$.
\end{thm}
\begin{proof}
Theorem~\ref{dirsum} and the discussion
at the end of Section~\ref{qhs} tell
 that $\omega \in
 {}_B \M^A_{B,S^2}$ is of the form 
$N \Box_C A$ for some $N \in
 \M^C_{B,S^2}$. It follows that 
as a special case of (the proof of) 
Theorem~5.8 in \cite{bobr} there are
isomorphisms of $A$-$B$-bimodules 
$$
		  A \otimes_B \omega \simeq
		  A \otimes_B (N \Box_C A) \simeq
		  N \Box_C (A \otimes_B A) \simeq
 N \Box_C (C \otimes A) \simeq
		  N \otimes A,
$$
where the left $A$-action on $N \otimes
 A$ is given by multiplication in $A$
 and the right $B$ action on $N \otimes
 A$ is $(n \otimes a)b:=nb_{(1)} \otimes
 S^2(b_{(2)})$.
The second isomorphism (the mixed associativity of
 $\Box_C$ and $\otimes_B$) uses
 the right flatness of $A$ and
 the third is the Galois
 isomorphism for the algebra extension $B
 \subset A$ which is explicitly given by 
$$A \otimes_B
 A \rightarrow C \otimes A,\quad
 x \otimes_B y \mapsto \pi (y_{(1)})
 \otimes xy_{(2)}.
$$ 
It follows that there is a right $B$-linear
 isomorphism
$$
		  N \simeq (A \otimes_B
 \omega)/A^+ (A \otimes_B \omega),\quad
 A^+=\mathrm{ker}\, \varepsilon.
$$
But we also have $A$-$B$-bimodule
isomorphisms 
$$
		  A \otimes_B \omega=
 A \otimes_B \mathrm{Ext}^{\mathrm{dim} (B)}_{B^e}(B,B \otimes B) \simeq 
		  \mathrm{Ext}^{\mathrm{dim} (B)}_{B^e}(B,B \otimes
 A) \simeq A_\sigma
$$
by Lemmata~\ref{blemma}
 and~\ref{geschrei}. Together this shows
 that as $B$-modules we have
$$
		  N \simeq
 \mathrm{Ext}^{\mathrm{dim}(B)}_B(k,B),
$$
and a coaction on the ground field is
 given by a group-like element $g \in C$ as
$$
		  k \ni \lambda \mapsto \lambda
 \otimes g \in k \otimes C 
$$
that has to obey (\ref{knights}) in
 order to define an object in $\M^C_{B,S^2}$.
The result follows now by the definition of
 $N \Box_C A$.
\end{proof}
 
\subsection{The Hopf-Galois case}\label{hoga} 
As we have recalled in the introduction,
the assumption $B^+A=A B^+$
means that $\pi : A \rightarrow C=A/B^+A$ is a Hopf
algebra quotient. Hence $\M^C$
is a monoidal category, where $M \otimes N$
is for $N,M \in \M^C$ the tensor product
over $k$ equipped with the coaction
$$
		  m \otimes n \mapsto m_{(0)} \otimes
		  n_{(0)} \otimes m_{(1)}n_{(1)}.
$$
Furthermore, any $M \in \M^C$ is
canonically an object in
$\M^C_{B,\mathrm{id}}$ if $B$ acts
trivially (through $\varepsilon$) from
the right. Hence
$M \Box_C A$ is canonically an 
object in ${}_B \M^A_B:={}_B \M^A_{B,\mathrm{id}}$
with $B$-bimodule structure 
\begin{equation}\label{assbim}
		  x (m \otimes a)y:=m \otimes
		  xay,\quad m \in M,a \in A,x,y
		  \in B,
\end{equation} 
and  
with respect to this
bimodule structure we have (this
generalises to any
faithfully flat Galois extension of an
algebra $B$ by a Hopf algebra $C$) 
\begin{equation}\label{wuetend}
		  (M \Box_C A) \otimes_B (N \Box_C A) \simeq
		  (M \otimes N) \Box_C A
\end{equation} 
as $B$-bimodules. Any group-like element
$g$ of $C$ is now invertible with
inverse $g^{-1}=S(g)$, and
Theorem~\ref{schmaerz} immediately gives:
\begin{cor}\label{qaz}
Retain all assumptions and notation
from Theorem~\ref{schmaerz} and
assume in addition $AB^+=B^+A$. Then
$\sigma (B) \subset B$, and if we consider
$$
		  \bar \omega:=\{a \in A\,|\,\pi
 (a_{(1)}) \otimes a_{(2)}=g^{-1}
 \otimes a\}
$$
as a $B$-bimodule via (\ref{assbim}), then we have
a $B$-bimodule isomorphism
$$
		  \bar \omega \otimes_B \omega
 \simeq B_\sigma.
$$
\end{cor}
\begin{proof} 
Take in (\ref{wuetend}) for $M$ the
 ground field $k$ with the $C$-coaction
 given by $\lambda \mapsto \lambda
 \otimes g^{-1}$ and for $N$ the same
 but with $g$ instead of $g^{-1}$.
Then we get $\bar \omega \otimes_B
 (N \Box_C A) \simeq B$ as
 $B$-bimodules. The $B$-bimodule
 $\omega$ is obtained from $N
 \Box_C A$ by twisting the right
 $B$-action by $\sigma$, so the claim follows.
\end{proof}

The fact that $\sigma (B) \subset B$ is
probably the most unexpected observation
here. It illustrates how restrictive
(\ref{knights}) is especially
for $AB^+=B^+A$ since it can in this
case be multiplied from the left by
$g^{-1}$ to give
$$
		  \pi (\sigma (b))=\chi (b) \pi
 (1)
$$ 
for all $b \in B$, and from this we indeed also compute
directly that
\begin{eqnarray}
\pi (\sigma (b)_{(1)}) \otimes \sigma
 (b_{(2)})
&=& \pi (\chi (b_{(1)})S^2(b_{(2)}))
\otimes S^2(b_{(3)}) \nonumber\\ 
&=& \pi (\sigma (b_{(1)}))
\otimes S^2(b_{(2)}) \nonumber\\ 
&=& \pi (1) \otimes \chi
 (b_{(1)})S^2(b_{(2)}) \nonumber\\ 
&=& \pi (1) \otimes \sigma
 (b),\nonumber
\end{eqnarray} 
hence $\sigma (b) \in B$ by Theorem~\ref{dirsum}.

We now want to show that in fact $\sigma
(B)=B$. For this we need a small
digression about characters and  
the following basic remark:
\begin{lemma}\label{berggg}
If $B \subset A$ is a quantum
homogeneous space and $AB^+=B^+A$, then
we have $S^2(B)=B$.
\end{lemma}
\begin{proof}
Koppinen's 
$S(AB^+)=B^+A$ (\cite{ms}, Lemma~1.4) gives 
$$S^2(B^+A)=S^2(AB^+)=S(B^+A)=S(AB^+)=B^+A$$
and hence for all $b \in B$
$$
		  \pi (S^{\pm 2}(b)_{(1)}) \otimes
 S^{\pm 2}(b)_{(2)}=
		  \pi (S^{\pm 2}(b_{(1)})) \otimes
 S^{\pm 2}(b_{(2)})=\pi (1) \otimes
 S^{\pm 2}(b),
$$ 
so $S^{\pm 2}(B) \subset B$ which
 implies $S^2(B)=B$.
\end{proof}
  
\subsection{Remarks on characters}\label{chrks} 
For a Hopf algebra $A$ the set 
$G:=\mathrm{Char}(A)$ of characters 
(algebra homomorphisms
$\gamma : A \rightarrow k$) is
canonically an affine group scheme
represented by the commutative Hopf
algebra $A/J(A)$, where 
$$
		  J(A):=\{a \in A\,|\,\gamma (a)=0
\,\forall\,\gamma \in
\mathrm{Char}(A)\},
$$
and for a right coideal subalgebra $B
\subset A$ the set $X:=\mathrm{Char}(B)$ becomes an
affine $G$-scheme represented by
$B/J(B)$. The (right) $G$-action on $X$ 
is given like the group structure in $G$
by the canonical product on $\mathrm{Hom}_k(A,k)$
\begin{equation}\label{convolution}
		  (\varphi \psi)(a):=\varphi
		  (a_{(1)})\psi (a_{(2)}) ,\quad
		  \varphi,\psi \in \mathrm{Hom}_k(A,k),a \in A
\end{equation}  
for which $\varepsilon$ is the unit element.

The inclusion $B \rightarrow A$ induces a homomorphism 
$B/J(B) \rightarrow A/J(A)$, and the restriction of a
character from $A$ to $B$ is the dual
morphism $G \rightarrow X$. 
However, even for some well-behaved examples of
quantum homogeneous spaces 
(such as the Podle\'s sphere that we will 
define in Section~\ref{stpodl}) the map $G
\rightarrow X$ is not surjective, 
$B/J(B) \rightarrow A/J(A)$ is not
faithfully flat and not even injective, and the 
$G$-action on $X$ is not transitive.

But at least we can say the following:
\begin{thm}\label{transes}
If $\chi$ is a character on a quantum
 homogeneous space
$B \subset A$, 
$\beta : A \rightarrow B$ is a right
$B$-linear projection 
as in Theorem~\ref{dirsum} (2) and
we define
$$ 
		  \gamma : A
 \rightarrow k, \quad a \mapsto \chi (\beta (S^{-1}(
		  a))),
$$ 
then we have
$$
		  \chi \gamma = \varepsilon 
$$  
as functionals on $B$.
\end{thm}
\begin{proof}
This follows by straightforward computation: 
\begin{eqnarray}
		  (\chi \gamma)(b)
&=&  \chi (b_{(1)})\chi
 (\beta(S^{-1}(b_{(2)})))=
  \chi (\beta(S^{-1}(b_{(2)})))\chi (b_{(1)}) \nonumber\\ 
&=&
 \chi(\beta(S^{-1}(b_{(2)}))b_{(1)})=
\chi(\beta(S^{-1}(b_{(2)})b_{(1)}))
\nonumber\\ 
&=& \chi (\beta (\varepsilon (b)))=
 \varepsilon (b),\nonumber
\end{eqnarray} 
where we used the
properties of $\chi$ and $\beta$ and
the fact that in every Hopf algebra
with bijective antipode we have 
$$
		  S^{-1}(a_{(2)})a_{(1)}=S^{-1}(S(S^{-1}(a_{(2)})a_{(1)}))=
S^{-1}(a_{(1)}S(a_{(2)}))=\varepsilon
 (a)$$
for all $a \in A$ since $S$ is always an
 algebra antihomomorphism.
\end{proof}

Note that $\gamma$ is in general not a character on
$A$, though.

\subsection{The proof of Theorem~\ref{mainneu}}\label{director} 
Theorem~\ref{transes} implies:
\begin{cor}\label{wsx}
If $\chi$ is a character on a quantum
 homogeneous space $B \subset A$, then
 the algebra homomorphism $\sigma : B \rightarrow A$ 
given by
$$
		  \sigma (b):=\chi (b_{(1)})S^2(b_{(2)})
$$
is injective. If $AB^+=B^+A$ 
and $\sigma (B) \subset
 B$, then $\sigma (B)=B$.
\end{cor}
\begin{proof}
An explicit left inverse of $\sigma$ is given by 
$$
		  \sigma^{-1} : A \rightarrow
 A,\quad
 \sigma^{-1}(a):=\gamma (S^{-2}(a_{(1)}))S^{-2}(a_{(2)}),
$$ 
where $\gamma$ is as in
 Theorem~\ref{transes}. 
Under the additional assumption $AB^+=B^+A$ we have
$S^2(B)=B$ (Lemma~\ref{berggg}),
so $\sigma (B) \subset B$ implies
$$\hat \sigma (b):=\chi
 (b_{(1)})b_{(2)}=S^{-2}(\sigma(b)) \in
 B$$ 
for $b \in B$. Now abbreviate for a
 given $b \in B$
$$
		  M:=\{\varphi
 (b_{(1)})b_{(2)}\,|\,\varphi \in
 \mathrm{Hom}_k(B,k)\} \cap B.
$$ 
We have 
$$
		  \hat \sigma (\varphi
 (b_{(1)})b_{(2)})=
		  \varphi (b_{(1)})\chi (b_{(2)})
 b_{(3)}=
		  (\varphi \chi)(b_{(1)})b_{(2)}
 \in M,
$$
that is, $\hat \sigma(M) \subset
 M$. Since $\sigma$ and hence $\hat
 \sigma$ has been shown already to be
 injective, $\hat \sigma|_M$ is
 bijective since $\mathrm{dim}_k(M)<\infty$ 
 (if $\Delta (b)=\sum_{-=0}^{n} x_i
 \otimes y_i$, then $M$ is spanned by the
 $y_i$). Furthermore, $b=\varepsilon
 (b_{(1)})b_{(2)} \in M$, so $b \in
 \mathrm{im}\, \hat \sigma$. Thus $b \in
 \mathrm{im}\, \hat
 \sigma$ for arbitrary $b \in B$ and
 hence also $\sigma=S^2 \circ \hat
 \sigma$ is surjective. 
\end{proof}

\begin{proof}[Proof of Theorem~\ref{mainneu}]
We have constructed in
Corollary~\ref{qaz} a $B$-bimodule
$\bar \omega$ with $\bar \omega
 \otimes_B \omega \simeq B_\sigma$, 
where $\sigma (b)=\chi
 (b_{(1)})S^2(b_{(2)})$. Corollary~\ref{wsx} 
shows that $\sigma$ is under the
 assumptions of Theorem~\ref{mainneu} an
 automorphism of $B$. This implies 
$$
		  B_\sigma \simeq {}_{\sigma^{-1}} B
$$
as bimodules, where ${}_{\sigma^{-1}} B$
is $B$ as right module but the left
 action is twisted by $\sigma^{-1}$ (the
 isomorphism is given by $\sigma^{-1}$). 
Hence we get
$$
		  {}_\sigma\bar \omega \otimes_B
 \omega \simeq B.
$$
To see that we also have 
$$
		  \omega \otimes_B {}_\sigma\bar \omega \simeq B
$$
note that we know
$$
		  \bar \omega \otimes_B
 \omega_{\sigma^{-1}} \simeq B,
$$
and $\omega_{\sigma^{-1}} \in {}_B
 \M^A_B$. Applying the monoidal functor 
$M \mapsto M/B^+M$ gives the
 corresponding $g^{-1}g=1$ in $C$, but
 here we also have $g
 g^{-1}=1$. Retranslating this into Hopf
 bimodules yields
\begin{equation*}
		  \omega \otimes_B {}_\sigma \bar
 \omega \simeq \omega_{\sigma^{-1}} \otimes
 \bar \omega \simeq B.\qedhere
\end{equation*}
\end{proof}

\subsection{A smoothness criterion}\label{criterium}
We mention here a
useful tool for proving the 
smoothness of $B \subset A$. 
The key remark
is \cite{mccr}, Theorem~7.2.8:
\begin{thm}
Let $B \subset A$ be a ring extension such that 
$B$ is a direct summand in $A$ as a $B$-bimodule. Then 
$\mathrm{gl.dim}(B) \le \mathrm{gl.dim} (A)+
\mathrm{proj.dim}_B(A)$.
\end{thm}

Together with Theorem~\ref{dirsum} this implies for
example:
\begin{cor}\label{moninervt}
If $B \subset A$ is a quantum homogeneous space 
and $A$ is commutative, then 
$\mathrm{gl.dim}(B) \le \mathrm{gl.dim} (A)$.
\end{cor}

But also in many noncommutative examples it will happen
that the decomposition in
Theorem~\ref{dirsum},(2) is actually
a decomposition of bimodules:

\begin{lemma}\label{dutz}
Let $B \subset A$ be a quantum homogeneous space and 
assume that $A/B^+A$ is cosemisimple with 
$A B^+=B^+ A$. Then 
$\mathrm{gl.dim}(B) \le \mathrm{gl.dim} (A)$.
\end{lemma}
\begin{proof}
As remarked above, the condition $A B^+=B^+ A$ means
that $C=A/B^+A$ is a Hopf algebra quotient of
$A$. The cosemisimplicity can be
 characterised as the existence of a
 (unique) 
functional $h : C \rightarrow k$ satisfying 
$$h(1)=1,\quad
 h(c_{(1)})c_{(2)}=c_{(1)}h(c_{(2)})=h(c),\quad
 c \in C,$$
see e.g.~\cite{chef},
 Theorem~13 in Section~11.2.1, and it is easily verified that 
$$
		  \beta : A \rightarrow A,\quad
 a \mapsto h(\pi(a_{(1)}))a_{(2)}
$$ is 
then a $B$-bilinear projection from $A$ onto 
$B \subset A$.   
\end{proof}

Note that the assumption of
cosemisimplicity of $C$ can be weakened, it suffices
that there is
a total integral $h : C \rightarrow A$
in the sense of \cite{doit} whose image
commutes with $B \subset A$ as in [ibid.],
Proposition (1.7)(b).

\begin{cor}\label{wirklich}
If $B \subset A$ is as in Lemma~\ref{dutz}
and $A^e$ is left Noetherian with 
$ \mathrm{gl.dim} (A^e)<\infty$, then $B$ is smooth. 
\end{cor}
 \begin{proof}
If $B \subset A$ is as in the lemma, then
so is $B^e \subset A^e$, hence the lemma gives 
$\mathrm{gl.dim}(B^e) \le
\mathrm{gl.dim}(A^e)$. Therefore, 
$\mathrm{gl.dim}(A^e)<\infty$ implies that 
the left $B^e$-module
$B$ has finite projective dimension. 
Finally, the left Noetherianity of $A^e$ implies that of $B^e$ 
(apply 
$A^e \otimes_{B^e} \cdot$ to an ascending chain of left
ideals in $B^e$ and use faithful flatness). 
Therefore, the projective dimension 
of the finitely generated $B^e$-module $B$ will
coincide with its projective dimension in the
category of finitely generated $B^e$-modules.
\end{proof}

\subsection{Some (counter)examples}\label{countex}
Before entering the discussion of the
Podle\'s sphere let us mention here three simpler
but instructive examples.

First of all, every Hopf subalgebra $B
\subset A$ is
in particular a right coideal
subalgebra. If $A=U(\mathfrak{g})$ and
$B=U(\mathfrak{h})$ are universal
enveloping algebras of
finite-dimensional Lie algebras
$\mathfrak{h} \subset \mathfrak{g}$, 
then the Poincar\'e-Birkhoff-Witt
theorem says that $A$ is free over $B$
and hence faithfully flat. 
 However, even the basic example 
of the Borel subalgebra $\mathfrak{h}
:=\mathfrak{b}_+$ in
$\mathfrak{g}:=\mathfrak{sl}(2,k)$
behaves rather badly: the 
characters of $\mathfrak{b}_+$ are in
bijection with $k$ but only one of them
(the counit) extends to $A$. The
dualising bimodule of $A$ is
$A$ without any twist $\sigma$, but that
of $B$ is of the form $B_\sigma$ for a
nontrivial automorphism (see \cite{bz}, this 
example was suggested by Ken
Brown to me). Note also that
$AB^+=B^+A$, but $B$ is not as a
$B$-bimodule a direct summand in $A$.

Secondly, consider $B=k[y]$ and for $A$
the Hopf algebra obtained by adding a
generator $x$ satisfying 
$$
		  x^2=1,\quad xy=-yx,
$$
so $A$ is the smash (aka crossed or
semidirect) product $B \rtimes
\mathbb{Z}_2$ of $B$ by
the automorphism that sends $y$ to
$-y$. The Hopf algebra structure is
given by
$$
		  \Delta (x) = x \otimes x,\quad
		  \Delta (y)=1 \otimes y+y \otimes
		  x,
$$
$$
		  \varepsilon (x)=1,\quad
		  \varepsilon (y)=0,\quad
		  S(x)=x,\quad S(y)=-yx.
$$ 
The monomials $\{y^i,xy^i\,|\,i \ge 0\}$
form a $k$-vector space basis of $A$, so
$A$ is free over $B$ with basis
$\{1,x\}$. In particular, $B \subset A$
is faithfully flat. 
In this example one can verify directly that 
$H^i(B,M) \simeq H_{1-i}(B,M)$ for all
$B$-bimodules $M$, and that 
$B$ is Gorenstein with $\chi =
\varepsilon$. However,  
$\sigma (b)=\chi (b_{(1)})S^2(b_{(2)})=S^2(b)$ 
is not the identity automorphism since
$$
		  S^2(y)=-S(yx)=-S(x)S(y)=xyx=-y.
$$

These examples show that even if $B$ satisfies
Poincar\'e duality it can be difficult
to read off $\omega$ from the Hopf-algebraic
data given. In particular it can happen
that the dualising bimodules of both $A$
and $B$ are of the form $A_\sigma$ and
$B_\tau$, but one can have $\tau =
\mathrm{id}_B,\sigma \neq
\mathrm{id}_A$ or conversely $\sigma
\neq \mathrm{id}_A,\tau =
\mathrm{id}_B$.

Finally, we would like to mention that the cusp $X
\subset k^2$ given by the equation
$x^2=y^3$ is also a quantum homogeneous space
although it is surely not a homogeneous
space of an algebraic group since it is
not smooth. The ambient Hopf algebra is 
again a skew-polynomial ring $A=B \rtimes
\mathbb{Z}$, $B=k[X]$, that is denoted by 
$B(1,1,2,3,q)$ in \cite{kenny7},
Construction~1.2. Therein the notation
is exactly the opposite of ours, their
$A$ is our $B$ and vice versa.  

\section{Application}\label{applikazie} 
\subsection{The standard Podle\'s sphere}\label{stpodl} 
For the rest of the paper we fix $k=\mathbb{C}$,
$q \in k^*$ is 
not a root of unity, and
$A$ is the standard quantised
coordinate ring of $SL(2,k)$ (see
e.g.~\cite{chef} for background information). This is the
Hopf algebra with
algebra generators $a,b,c,d$, defining relations
$$
		  ab=qba, \quad
		  ac=qca, \quad
		  bc=cb,\quad
		  bd=qdb, \quad
		  cd=qdc,
$$
$$
		  ad-qbc=1, \quad
		  da - q^{-1} bc=1
$$
and the coproduct, counit, and antipode determined by 
$$ 
		  \Delta(a)=a \otimes a + b \otimes c,\quad 
		  \Delta(b)=a \otimes b+b \otimes d,
$$
$$
		  \Delta(c)=c \otimes a+d \otimes c,\quad 
		  \Delta(d)=c \otimes b+d \otimes d,
$$
$$
		  \varepsilon(a)=\varepsilon(d)=1,\quad 
		  \varepsilon(b)=\varepsilon(c)=0,
$$
$$
		  S(a)=d,\quad S(b)=-q^{-1}b,\quad S(c)=-qc,\quad S(d)=a.
$$

It follows from these relations that there is a unique
Hopf algebra quotient 
$$
		  \pi : A \rightarrow C:=k[z,z^{-1}],\quad
		  \pi (a)=z,\quad \pi(d)=z^{-1},\quad
		  \pi (b)=\pi (c)=0,
$$
where the Hopf algebra structure 
of $ k [z,z^{-1}]$ is determined by 
$\Delta (z)=z \otimes z$, that is, $C$ is the
coordinate ring of $T=k^*$, and the map 
$ \pi $ would correspond for $q=1$ to the embedding of
$T$ as a maximal torus into $SL(2,k)$.

By Corollary~\ref{schneller}, $ \pi $ gives rise
to a quantum homogeneous space $B$ as in
(\ref{edwin}). This subalgebra deforms 
the coordinate ring of 
$T \setminus \! SL(2,k)$ and was
discovered by Podle\'s
\cite{Po} and hence is referred to by
most authors as the (standard)
Podle\'s quantum sphere. The elements 
$$	  
		  y_{-1}:=
		  ca,\quad
		  y_0:= bc,\quad
		  y_1:= bd
$$
generate $B$ as an algebra, and $B$ can be
characterised abstractly as the algebra with three
generators $y_{-1},y_0,y_1$ and 
defining relations
\begin{equation}\label{tkkg}
	y_0y_{\pm 1} = q^{\pm 2} y_{\pm 1}y_0,\quad
	y_{\pm 1}y_{\mp 1} = 
	 q^{\mp 2} y_0^2 + q^{\mp 1} y_0,  
\end{equation} 
see \cite{dijkhuizen,mas,Po}.

\subsection{The Koszul resolution of the $B$-module
 $k$}\label{koszu} 
We will construct a free resolution of
the $B$-module $k$ (with action given by
$\varepsilon$) by using the probably simplest case of 
Priddy's noncommutative Koszul
resolutions \cite{priddy}:
\begin{lemma}\label{koskom}
Let $B$ be a $k$-algebra and assume
 $z_{\pm 1} \in B$ are such that
\begin{enumerate}
\item $z_{-1}z_1=\lambda z_1z_{-1}$ for some $\lambda \in k$,  
\item $az_{-1}=0$ implies $a=0$ for all $a \in B$
	  and
\item $\nu (az_1)=0$ implies $\nu (a)=0$ for all
		$\nu (a):=a \,\mathrm{ mod }\, Bz_{-1} \in B/Bz_{-1}$. 
\end{enumerate} 
Then the chain complex $K_\bullet:=K_\bullet(z_1,z_{-1})$
given by
$$
		  0 \rightarrow B \rightarrow B
		  \oplus B \rightarrow B
		  \rightarrow 0
$$
with nontrivial boundary maps
$$
		  a \mapsto (az_{-1},-\lambda az_1),\quad
		  (b,c) \mapsto bz_1+cz_{-1}
$$
is a free resolution of 
$B/I$, $I:=Bz_1+Bz_{-1}$.
\end{lemma}
\begin{proof}
We clearly have 
$H_0(K)=B/I$ by very definition
and  
$H_2(K)=0$ by assumption (2).
Now consider the subcomplex 
$$
		  \tilde K := 0 \rightarrow B
		  \rightarrow Bz_{-1} \oplus B
		  \rightarrow Bz_{-1} \rightarrow 0
$$
of $K$ and the quotient complex
 $K/\tilde K$ which is of
the form
$$
		  0 \rightarrow 0
		  \rightarrow B/Bz_{-1} 
		  \rightarrow B/Bz_{-1} \rightarrow 0.
$$
Its one nontrivial boundary map map is
$$
		  B/Bz_{-1} \rightarrow B/Bz_{-1},\quad
		  \nu (a) \mapsto \nu (az_1),
$$
so assumption (3) means $H_1(K/\tilde
 K)=0$. Furthermore, we have $H_1(\tilde
 K)=0$: a cycle is an element
$(bz_{-1},c) \in Bz_{-1} \oplus B$ with
$$
		  0=bz_{-1}z_1+cz_{-1}=(\lambda bz_1+c)z_{-1},
$$
so assumption (2) gives  
$c=-\lambda bz_1$, hence
$
		  (bz_{-1},c)=(bz_{-1},-\lambda bz_1)
$
is a boundary. Considering the long exact
homology sequence derived from the short exact sequence 
$		  0 \rightarrow \tilde K
		  \rightarrow K \rightarrow
		  K/\tilde K \rightarrow 0
$
now yields $H_1(K)=0$.
\end{proof}

From now on let $B$ be again the Podle\'s
sphere. Then the above gives:

\begin{thm}
The left $B$-module $ k $ admits a free resolution
of the form $K_\bullet(z_1,z_{-1})$ with
$z_{\pm 1}:=y_{\pm 1}+y_0$, $\lambda :=q^2$.
\end{thm} 
\begin{proof}
It is easily seen that $B^+:=B \cap \mathrm{ker}\,
\varepsilon$ is generated as a left ideal by the
elements $y_n$. But since one has
$q^{-1}y_{-1}(y_1+y_0)
-q y_0(y_{-1}+y_0)=
y_0$,
$B^+$ is in fact generated as a left
 ideal by the two elements
$z_{\pm 1}$.

One verifies directly that
$z_{-1}z_1=q^2z_1z_{-1}$ which is
 assumption (1) in
Lemma~\ref{koskom}. Secondly, $B$ is 
a domain (see
 e.g.~\cite{bavula}), so assumption (2)
 holds as well. For (3) we turn $B$ into a
 $\mathbb{Z}$-graded algebra
by assigning to $y_i$ the degree $i$
which is compatible with the defining
relations (\ref{tkkg}). Then we have
$$
		  B=\bigoplus_{j \in \mathbb{Z}}
 B_j,\quad B_iB_j \subset B_{i+j},\quad
 B_j:=\mathrm{span}_k\{e_{ij}\,|\,
		  i \ge 0\},
$$
where 
$$
	e_{ij}:=\left\{
	\begin{array}{ll}
	y_0^iy_1^j \quad& j \ge 0,\\
	y_0^iy_{-1}^{-j}\quad& j<0,
	\end{array}\right.
	i \in \mathbb{N}_0,j \in \mathbb{Z},  
$$
and these form a vector space basis of $B$. Under 
 $\nu : B \rightarrow B/Bz_{-1}$ we have
\begin{eqnarray}
		  \nu (y_0^iy_{-1}^j)
&=& y_0^iy_{-1}^{j-1} \nu
(y_{-1})=-y_0^iy_{-1}^{j-1} \nu (y_0)
\nonumber\\ 
&=&-y_0^i \nu (y_{-1}^{j-1}y_0)
=-q^{2(j-1)}y_0^i \nu (y_0y_{-1}^{j-1})
\nonumber\\ 
&=&-q^{2(j-1)}y_0^{i+1} y_{-1}^{j-2}\nu (y_{-1})
\nonumber\\
&=&q^{2(j-1)}y_0^{i+1} y_{-1}^{j-2}\nu (y_0)
\nonumber\\
&=&q^{2(2j-1-2)}y_0^{i+2} y_{-1}^{j-3}\nu (y_{-1})
\nonumber\\
&=&-q^{2(2j-1-2)}y_0^{i+2} y_{-1}^{j-3}\nu (y_0)
\nonumber\\
&=& \ldots \nonumber\\ 
&=& (-1)^jq^{2((j-1)j-1-2-\ldots-(j-1))}y_0^{i+j-1} y_{-1}^{j-j}\nu (y_0)
\nonumber\\
&=& (-1)^jq^{(j-1)j}\nu (y_0^{i+j}).
\nonumber
\end{eqnarray} 
Similarly we have for $i>0,j>0$
\begin{eqnarray}
\nu (y_0^iy_1^j)
&=& q^{2j}y_0^{i-1}  \nu (y_1^jy_0) =
q^{2j} y_0^{i-1} y_1^j \nu (y_{-1})
\nonumber\\ 
&=& q^{2j} y_0^{i-1} y_1^{j-1} \nu (q^{-2}y_0^2+q^{-1}y_0)
\nonumber\\ 
&=& q^{2j-2} y_0^{i-1} y_1^{j-1} \nu
 (y_0^2)+
q^{2j-1} y_0^{i-1} y_1^{j-1} \nu (y_0)
\nonumber\\ 
&=& q^{-2j+2}\nu
 (y_0^{i+1} y_1^{j-1})+
q \nu (y_0^iy_1^{j-1})
\nonumber\\ 
&=& q^{-2j+2}(q^{-2j+4}\nu
 (y_0^{i+2} y_1^{j-2})+
q \nu (y_0^{i+1}y_1^{j-2}))
\nonumber\\ 
&& +q(q^{-2j+4}\nu
 (y_0^{i+1} y_1^{j-2})+
q \nu (y_0^iy_1^{j-2}))
\nonumber\\ 
&=& q^{-4j+6}\nu
 (y_0^{i+2} y_1^{j-2})+
q^{-2j+3}(1+q^2)\nu (y_0^{i+1}y_1^{j-2})+
q^2 \nu (y_0^iy_1^{j-2})
\nonumber\\ 
&=& q^{-4j+6}
 (q^{-2j+6}\nu
 (y_0^{i+3} y_1^{j-3})+
q \nu (y_0^{i+2}y_1^{j-3})) \nonumber\\ 
&&+q^{-2j+3}(1+q^2)(q^{-2j+6}\nu
 (y_0^{i+2} y_1^{j-3})+
q \nu (y_0^{i+1}y_1^{j-3}))
\nonumber\\ 
&&+
q^2 (q^{-2j+6}\nu
 (y_0^{i+1} y_1^{j-3})+
q \nu (y_0^iy_1^{j-3}))
\nonumber\\ 
&=& q^{-6j+12}\nu
 (y_0^{i+3} y_1^{j-3})+
q^{-4j+7} (1+q^2+q^4)\nu (y_0^{i+2}y_1^{j-3}) \nonumber\\ 
&&+q^{-2j+4}(1+q^2+q^4)\nu (y_0^{i+1}y_1^{j-3})
+
q^3 \nu (y_0^iy_1^{j-3})
\nonumber\\ 
&=& \ldots \nonumber\\ 
&=& \sum_{r=0}^j
 q^{(-2r+1)j+r^2}
		 \left(\begin{array}{c} j \\ r \end{array}\right)_q
		 \nu (y_0^{i+r}) 
\nonumber
\end{eqnarray}
where we abbreviated
$$
\left(\begin{array}{c} j \\ r
		\end{array}\right)_q:=
1+q^2+q^4+ \ldots + q^{2\left(\begin{array}{c} j \\ r
		\end{array}\right)-2}. 
$$
Thus we have
$$
		  B/Bz_{-1}=\mathrm{span}_k\{\nu(y_0^{i+1}),\nu
(y_1^i)\,|\,i
 \ge 0\}.
$$

These residue classes 
are also linearly independent: assume that
\begin{equation}\label{cococo6}
		  \sum_{i \ge 0} \lambda_i \nu(y_0^i)+
		  \sum_{j \ge 0} \mu_j \nu(y_1^{j+1})=0
\end{equation} 
in $B/Bz_{-1}$. One easily checks that
$$
		  B/(Bz_{-1}+By_0)=B/(By_{-1}+By_0)
$$ 
is an algebra quotient of $B$
(i.e.~that $By_0+By_{-1}$ is a two-sided ideal in
 $B$) and
that it is as such isomorphic to 
the polynomial ring generated
by the residue class of
 $y_1$. Hence the residue classes of
 $y_1^j$ are linearly independent 
in this quotient of $B/Bz_{-1}$. Considering the image of
(\ref{cococo6}) therein thus gives
$$
		  \mu_j=0\quad\forall\,j \ge 0.
$$ 
We are left with 
\begin{equation}\label{tatort}
		  \sum_{i \ge 0} \lambda_i \nu(y_0^i)=0
		  \quad\Leftrightarrow\quad 
		  \sum_{i \ge 0} \lambda_i y_0^i=
		  az_{-1}
\end{equation} 
for some $a \in B$. But 
$\sum_{i \ge 0} \lambda_i y_0^i$ 
is homogeneous of degree
 $0$, $B$ is a domain, and 
$z_{-1}$ is not homogeneous, so the right
hand side can not be homogeneous unless $a =
 0$: if 
$$
		  a=a_{j_0}+\ldots+a_{j_n},\quad
		  a_{j_i} \in B_{j_i} \setminus \{0\},\quad 
		  j_0<\ldots<j_n
$$
is the decomposition of $a$ into
 homogeneous components, then 
$az_{-1}$ has a nonzero component $a_{j_0}y_{-1}$ 
in degree $j_0-1$ and a nonzero component $a_{j_n}y_0$
in degree $j_n$. Thus $a=0$ and since
 the $y_0^i$ are linearly independent in
 $B$ it follows that also 
$$
		  \lambda_i=0\quad\forall\,i \ge 0.
$$

Now we compute the action of he map
\begin{equation}\label{duze}
		  \zeta : B/Bz_{-1} \rightarrow
 B/Bz_{-1},\quad
 \nu (a) \mapsto \nu (az_1)
\end{equation} 
on the basis vectors. We get for $i>0$
\begin{eqnarray}
		  \nu (y_0^iz_1)
&=& \nu (y_0^iy_1)+\nu (y_0^{i+1})
\nonumber\\ 
&=& q^2 y_0^{i-1} \nu (q^{-2}y_0^2+q^{-1}y_0)+\nu (y_0^{i+1})
\nonumber\\
&=&  2\nu (y_0^{i+1})+q \nu (y_0^i)
\nonumber
\end{eqnarray}
and for $j \ge 0$
\begin{eqnarray}
		  \nu (y_1^jz_1)
&=& \nu (y_1^{j+1})+\nu (y_1^jy_0)
\nonumber\\ 
&=& \nu (y_1^{j+1})+q^{-2j}\nu (y_0y_1^j)
\nonumber\\ 
&=& \nu (y_1^{j+1})+q^{-2j}
\sum_{r=0}^j
 q^{(-2r+1)j+r^2}
		 \left(\begin{array}{c} j \\ r \end{array}\right)_q
		 \nu (y_0^{1+r}). 
\nonumber
\end{eqnarray} 
So if we abbreviate
$$
		  V_j:=\mathrm{span}_k\{\nu
 (y_0),\ldots,\nu(y_0^{j+1}),\nu(1),\nu(y_1),\ldots,\nu(y_1^j)\},
$$
then we have 
$$
		  B=\bigcup_{j \ge 0} V_j,\quad
		  \zeta (V_j) \subset V_{j+1}
$$
and $\zeta|_{V_j}$ is represented with
 respect to our basis by a matrix of the
 form
$$
\left(\begin{array}{cccccc} 
q &  & & * & \ldots & *\\
 2 & \ddots & & & \ddots & \vdots \\
& \ddots & q & & & * \\
& & 2 & 0 & \ldots & 0\\
& &  & 0 & &\\
& &  & 1 & \ddots &\\
& &  &  & \ddots& 0\\
& &  &  & & 1
		\end{array}\right),  
$$ 
where the $*$ denote nonzero
 entries and all other entries vanish.
Hence $\zeta$ is evidently injective
 (composing $\zeta|_{V_j}$ with the
 canonical projection onto
 $V_{j+1}/\mathrm{span}_k\{\nu(y_0),\nu(1)\}$
 yields an isomorphism of determinant $2^j$)
which is assumption
(3) of Lemma~\ref{koskom}.
\end{proof}

\subsection{The Gorenstein condition} 
From the minimal resolution of $ k $
provided by the Koszul complex one
can read off $\mathrm{Ext}^n_B(k,B)$:
\begin{lemma}\label{scotchpie1}
One has
$\mathrm{Ext}^n_B(k,B)=0$ for $n \neq 2$ and 
$\mathrm{Ext}^2_B(k,B) \simeq k$. 
The resulting character $ \chi $ of $B$ is equal to
$ \varepsilon $.
\end{lemma}
\begin{proof}
Apply $\mathrm{Hom}_B(\cdot,B)$ to the Koszul
complex and identify 
$ \mathrm{Hom}_B(B,B) \simeq B$. This
gives the cochain complex
$$
		  0 \leftarrow B \leftarrow B \oplus B
		  \leftarrow B \leftarrow 0
$$
of right $B$-modules whose 
two nontrivial coboundary maps are given by
$$
		 f \mapsto 
		 (z_1f,z_{-1}f),\quad  
		  (f,g) \mapsto q^{-1} z_{-1}f-q z_1g.
$$
The exactness of this complex in degree 0 and 1
can be shown as the exactness of the
 Koszul complex using  
Lemma~\ref{koskom} (with $B$ replaced
 by $B^\mathrm{op}$).
In degree 2, the cohomology is $B$ divided by
the right ideal generated by 
$z_{\pm 1}$. The result follows since with
$qz_1 y_{-1}-q^{-1}z_{-1} y_0=y_0$ one
 easily deduces that this ideal is again
 $\mathrm{ker}\, \varepsilon$.
\end{proof}

Thus the relevant twisting automorphism is
$$
		   \sigma (f) = \chi (f_{(1)}) S^2(f_{(2)}) =
S^2(f)$$ 
which is explicitly given by
$$
		  \sigma (y_{-1})=q^2 y_{-1},\quad
 		  \sigma (y_0) = y_0,\quad 
		  \sigma (y_1) = q^{-2}y_1.
$$
Note this is also the restriction of Woronowicz's modular
automorphism (see e.g.~\cite{chef,tom}
for more information) to
$B$.

\subsection{The smoothness condition}
The smoothness of $B$ follows from 
Corollary~\ref{wirklich} since for this example 
$A^e \simeq k_q[SL(2) \times SL(2)]$
is left Noetherian with 
global dimension 4, see \cite{kenny2} and the
references therein.
 
\subsection{Determining $\omega$}
We now know that Theorem~\ref{mainneu}
applies to $B$ with $\mathrm{dim}(B)=2$,
and that $B$ acts trivially (via
$\varepsilon$) on
$\mathrm{Ext}_B^2(k,B) \simeq k$ so that
the automorphism $\sigma$ from
Theorem~\ref{schmaerz} equals
$S^2$. Equation (\ref{knights}) 
becomes trivial, so
$g$ therein could be any of
the group-like elements in 
$C=k[z,z^{-1}]$, that is, an arbitrary monomial 
$z^n$ for some $n \in \mathbb{Z}$. So according to
Theorem~\ref{schmaerz}, $\omega$ is
isomorphic as an object of 
${}_B \M^A_{B,S^2}$ to $\omega_{n,1}$,
where we define for
$m,n \in \mathbb{Z}$
$$
		  \omega_{n,m}:=\{a \in
		  A_{S^{2m}}\,|\,\pi(a_{(1)})
		  \otimes a_{(2)}=z^n \otimes a\}
		  \in {}_B \M^A_{B,S^{2m}}.
$$
As $B$-bimodules, we have isomorphisms
$$
		  \omega_{n,m} \simeq
		  (\omega_{n,0})_{S^{2m}} \simeq
		  (\omega_{n,0}) \otimes_B
		  B_{S^{2m}}=
		  \omega_{n,0} \otimes_B \omega_{0,m}
$$
and (as a special case of (\ref{wuetend}))
$$
		  \omega_{l,0} \otimes_B
		  \omega_{n,0} \simeq \omega_{l+n,0}.
$$
Finally, the $B$-bimodule isomorphism
$S^{2m} : {}_{S^{-2m}}A \rightarrow
A_{S^{2m}}$ 
restricts to a $B$-bimodule isomorphism
$$
		  {}_{S^{-2m}}(\omega_{n,0}) \simeq \omega_{n,m},
$$ 
and combining these three equations we
see that as $B$-bimodules we have
\begin{equation}{\label{fridge}}
		 \omega_{n,m} \otimes_B
		 \omega_{i,j} \simeq \omega_{n+i,m+j}. 
\end{equation}

Furthermore, we obtain by direct
computation:
\begin{lemma}\label{wryi}
One has 
$$
		  H^0(B,\omega_{i,j}) \simeq \left\{
\begin{array}{ll}
k\quad & i=2(m-j) \mbox{ for some }  0
 \le m \le 2j ,\\
0\quad &\mbox{otherwise.}
\end{array}\right.  
$$
\end{lemma}
\begin{proof}
The defining relations of
 $A$ imply that the monomials 
$$
	f_{lmn}:=\left\{
	\begin{array}{ll}
	a^lb^mc^n \quad& l \ge 0,\\
	d^{-l}b^mc^n\quad& l<0,
	\end{array}\right.
	l \in \mathbb{Z},m,n \in \mathbb{N}_0 
$$
form a vector space basis, that $A$ is a $\mathbb{Z}$-graded 
$B$-bimodule, 
$$
		  A=\bigoplus_{l \in
 \mathbb{Z}} A_l,\quad
 B_iA_lB_j \subset A_{i+j+l},\quad
 A_l:=\mathrm{span}_k\{f_{lmn}\,|\,m,n
 \ge 0\},
$$
and that
$$
		  A_l=\{f \in A\,|\,y_0f=q^{2l}fy_0\}.
$$
Furthermore, the explicit formulas for
the coproduct of the generators $a,b,c,d$ given in
 Section~\ref{stpodl} and the fact that $\pi
 (f_{lmn})=\delta_{m,0}\delta_{n,0}z^l$ show that
$$
		  \omega_{i,j}=\mathrm{span}_k\{f_{lmn}\,|\,i=l+m-n\}
 \subset A_{S^{2j}}.
$$
The zeroth Hochschild cohomology 
is by very definition isomorphic to the centre
of the coefficient bimodule (identify 
$\varphi \in \mathrm{Hom}_{B^e}(B,M)$
 with 
$\varphi (1) \in M$), and since
 $S^2(y_0)=y_0$, the last two equations imply
$$
		  H^0(B,\omega_{i,j}) \subset
 \omega_{i.j} \cap A_0=
 \mathrm{span}_k\{b^mc^n\,|\,m-n=i\}.
$$ 
And from $S^2(y_{\pm 1})=q^{\mp 2}y_{\pm
 1}$ and
$$
		  y_{\pm 1} b^mc^n=
		  q^{\mp(m+n)}b^mc^ny_{\pm 1}
$$
we deduce now that 
$$
		  H^0(B,\omega_{i,j}) =
 \mathrm{span}_k\{b^mc^n\,|\,m-n=i,m+n=2j\}.
$$ 
The claim follows by elementary arithmetics.
\end{proof}

Now we can finish the
proof of Theorem~\ref{appli}:
\begin{proof}[End of proof of Theorem~\ref{appli}]
By Theorem~\ref{mainneu} and
Hadfield's explicit computation of
$H_\bullet(B,B_{S^2})$ \cite{tom}
we have 
\begin{align}\label{tommy1}
		  H^0(B,(\omega^{-1})_{S^2})
&\simeq H^0(B,\omega^{-1} \otimes_B
		  B_{S^2}) \nonumber\\ 
&\simeq H_2(B,\omega \otimes_B
			\omega^{-1} \otimes_B
 B_{S^2}) \\
&\simeq H_2(B,B_{S^2}) \simeq
 k.\nonumber
\end{align} 
If $\omega = \omega_{n,1}$, then 
by (\ref{fridge}) we have 
$$
		  (\omega^{-1})_{S^2} \simeq \omega_{-n,0},
$$
and inserting this into Lemma~\ref{wryi}
yields $n=0$, so we have $\omega \simeq
 \omega_{0,1}=B_{S^2}$ 
as claimed in Theorem~\ref{appli}.
\end{proof}

\end{document}